\documentclass{tac}

\usepackage{amssymb,,mathrsfs,amsmath,enumerate,url,subcaption,mathbbol}
\usepackage[square]{natbib}

\usepackage[colorlinks=true]{hyperref}
\hypersetup{allcolors=[rgb]{0.1,0.1,0.4}}

\author{Tai-Danae Bradley and Juan Pablo Vigneaux}

\thanks{The authors thank Matilde Marcolli and Yiannis Vlassopoulos for engaging in numerous insightful discussions on category theory, language models, and linguistics that have influenced this article. We also thank the anonymous referee whose feedback helped to improve the article considerably.}

\address{SandboxAQ, Palo Alto, CA 94301, USA\\
Department of Mathematics, The Master's University, Santa Clarita, CA 91321, USA\\[5pt]
 Department of Mathematics, California Institute of Technology, Pasadena, CA 91125, USA\\
 Department of Linguistics, Northwestern University, Evanston, IL 60208, USA
}

\title[Categories of texts enriched by language models]{The magnitude of categories of texts enriched by language models}

\copyrightyear{2025}


\eaddress{tai.danae@math3ma.com\CR jpvigneaux@northwestern.edu}

\newcommand{\Int}{\operatorname{int}}

\DeclareMathOperator{\ob}{ob}
\newcommand{\Mag}{\operatorname{Mag}}

\newtheorem{proposition}{Proposition}

\newtheorem{corollary}{Corollary}

\newtheoremrm{definition}{Definition}
\newtheoremrm{remark}[theorem]{Remark}

\begin{document}

\maketitle

\begin{abstract}
    The purpose of this article is twofold. Firstly, we use the next-token probabilities given by a language model to explicitly define a category of texts in natural language enriched over the unit interval, in the sense of Bradley, Terilla, and Vlassopoulos. We consider explicitly the terminating conditions for text generation and determine when the enrichment itself can be interpreted as a probability over texts.  Secondly, we compute the M\"obius function and the magnitude of an associated generalized metric space of texts. The magnitude function of that space is a sum over texts (\emph{prompts}) of the $t$-logarithmic (Tsallis) entropies of the next-token probability distributions associated with each prompt, plus the cardinality of the model's possible outputs. A suitable evaluation of the magnitude function's derivative recovers a sum of Shannon entropies, which justifies seeing magnitude as a partition function. Following Leinster and Shulman, we also express the magnitude function of the generalized metric space as an Euler characteristic of magnitude homology and provide an explicit description of the zeroeth and first magnitude homology groups.
\end{abstract}

\section{Introduction}

In recent years, advances in large language models (LLMs) have brought renewed attention to mathematical structures underlying language. Building on the observation that LLMs acquire a great deal of semantic information through the statistics of co-occurrences of fragments of text, the first author together with Terilla and Vlassopoulos described a category-theoretical framework for mathematical structure inherent in text corpora in \citep{ECTL2022}. This involves a category of strings from a finite alphabet of symbols, with morphisms indicating substring containment; the strings correspond to texts or fragments of texts in a  language.  Statistical information is incorporated through an enrichment of this category over the unit interval by assigning a value $\pi(y|x)\in[0,1]$ to each pair of strings $x$ and $y$, which can intuitively be thought of as the probability that $y$ is an extension of a prompt $x$. By further embedding this enriched category of strings into an enriched category of copresheaves valued in the unit interval, a setting which contains rich mathematical structure, one can further explore semantic information via enriched analogues of categorical limits and colimits, which are akin to logical operations on meaning representations of texts \citep{ECTL2022}.  More recently, Liu \emph{et al.}~have provided experimental evidence supporting this semantic interpretation of the co-Yoneda embedding of texts in the enriched category of copresheaves \citep{liumeaning}. 

While an explicit construction of $\pi(y|x)$ was not given in \citep{ECTL2022}, we will show in this article that these values may in fact arise from next-token probabilities generated by a language model. For context: in natural language processing, texts are analyzed by  splitting them into tokens|which might be characters, words, or word fragments|and a \emph{language model} (LM) is an algorithm (possibly learned) that, for any finite sequence of \emph{tokens} (i.e.~a text fragment), assigns a probability to any given token being its continuation \cite[Ch.~3]{Jurafsky2025}. These are sometimes referred as \emph{autoregressive} or \emph{causal} LMs, particularly when compared with \emph{masked} LMs \cite[Ch.~11]{Jurafsky2025}. The simplest causal LMs are $n$-gram models, which make their prediction based on the last $(n-1)$-tokens of the prompt. The relatively recent large language models, which are deep learning models trained on vast amounts of data, have extended the context used for prediction to thousands or even millions of tokens. Most of today's popular LLMs, such as those in the GPT and Llama families, are autoregressive.  

Let us briefly review some basics and then state the  main results of this work. To start, it is standard for current language models to use special characters called the \emph{beginning-} and \emph{end-of-sentence} tokens, which we will denote by $\bot$ and $\dagger$, respectively. These tokens are usually not visible to the user but indicate to the model the boundary, so to speak, of a text, which is essential during training and when generating output. (In this way, a ``sentence'' may simply be thought of as a string of text built up token by token.) Additionally, the models have a \emph{cutoff} size, or maximum length limit, to prevent unreasonably long outputs or infinite loops. 

So in practice, for a tokenized user input $a$, the model is fed the prompt $x=\bot a$ and generates a probability distribution $p_x$ on its set of tokens. It then samples from that distribution to select a token $a_1$. If $a_1$ is the end-of-sentence token $\dagger$, the model outputs $a$ to the user. Otherwise, the string $xa_1$ is fed back into the model, which again generates a probability distribution $p_{xa}$ on its set of tokens. It samples from that distribution to select a token $a_2$, and the process repeats. If $a_2=\dagger$, the model outputs $aa_1$ to the user. Otherwise, it continues in this fashion until it either selects $\dagger$ or reaches the cutoff size. For convenience, if a string of tokens ends in $\dagger,$ then we will refer to it as a \emph{finished} text. Otherwise, it is \emph{unfinished}. This brief background motivates the earlier claim that the values from $\pi$ indeed form the hom-objects of a category enriched over the unit interval. Concretely, we will show the following main result:\\

\noindent \textbf{Every LM defines an enriched category} (see Proposition \ref{prop:LLM_defines_ecat}). \quad \textit{Every autoregressive language model defines a $[0,1]$-category whose objects are strings of tokens that begin with $\bot$ and may or may not end with $\dagger$, and whose hom-objects are given by $\pi$.} \\

Our consideration of beginning/end-of-sentence tokens as well as cutoff values differentiates this result from the work in \citep{GV2024}, where a similar enriched category-theoretical construction appears.
This will also allow us to prove that although the $\pi(y|x)$ are not \textit{literally} probabilities (for instance, we will see that $\pi(x|x)=1$ for each object $x$), for each unfinished text $x$ the function $\pi(-|x)$ can be regarded as a probability mass function on the set $T(x)$ of \emph{terminating states} of a prompt $x$, which is the set of all strings having $x$ as a prefix that either end with $\dagger$ or reach the cutoff size. Said differently, $T(x)$ is the set of all theoretically possible outputs of a model (including those with negligible probability), given the prompt $x$. This is another contribution of this work:\\

\noindent \textbf{Hom objects restrict to a probability mass function} (see Proposition \ref{prop:prob_terminal}). \quad \textit{Every autoregressive language model determines a probability mass function $\pi(-|x)$ on the set of terminating states of an unfinished text $x$.}\\

It was remarked in \citep{ECTL2022} and \citep{GV2024}  that one can alternatively enrich the category of strings over the extended non-negative reals $[0,\infty]$ by defining the hom-object between strings $x$ and $y$ to be $-\ln\pi(y|x)$. 
 \citet{lawvere73} observed that one can see a $([0,\infty],\geq,+,0)$-enriched category as a \emph{generalized metric space} by interpreting the hom-objects as  distances; the resulting distance function may be non-symmetric and degenerate but still satisfies the triangle inequality. In this article, we extend this geometric perspective by computing a fundamental invariant of enriched categories---their \emph{magnitude}. Magnitude generalizes the cardinality of a set and the Euler characteristic of a topological space, and it can be regarded as the ``size'' of an enriched category \citep{Leinster2008, Leinster2013}.  The computation of magnitude of ordinary metric spaces seen as $[0,\infty]$-enriched categories reveals rich connections with traditional invariants from integral geometry and geometric measure theory such as volume, capacity, dimension, and intrinsic volumes \citep{LeinsterMeckes}. 

In more detail, we study the generalized metric space $\mathcal M$ whose elements are strings $x,y,\ldots$ and whose  distances are given by $d(x,y)=-\ln \pi(y|x)$. We use the approach introduced by 
\citet{rota1964}, later extended by 
\citet{LeinsterShulman2021}, to compute the M\"obius coefficients of $\mathcal M$  (see Proposition \ref{prop:vigneaux}). Then we use these M\"obius coefficients to calculate the \emph{magnitude function} $\Mag(t\mathcal M)$, for $t>0$.\\ 

\noindent\textbf{Magnitude function associated with an LM} (see Proposition \ref{prop:magnitude_L}). \quad \textit{The magnitude function of the generalized metric space $\mathcal{M}$ associated with an autoregressive language model is equal to
\[\Mag(t\mathcal{M})=(t-1)\sum_{x\in\ob(\mathcal{M})\setminus T(\bot)}H_t(p_x) + \# (T(\bot)), \qquad t>0.\]}\\

\noindent In the above expression, 
\begin{enumerate}
    \item $\#(T(\bot))$ is the number of possible outputs of the model,
    \item $p_x$ is the probability distribution over tokens generated by the model prompted by $x$, and
    \item $H_t$ is the  $t$-logarithmic entropy: $H_t(p_1,\ldots,p_m) = (1-\sum_{i=1}^m p_i^t)/(t-1)$.\\
\end{enumerate}    

The entropies $(H_t)_{t\in (0,\infty)\setminus\{1\}}$ were first introduced by 
\citet{Havrda1967} and later popularized in physics by 
\citet{Tsallis1988}. Each  $H_t$ is positive and strictly concave; it vanishes when its argument  concentrates all the probability on a single token (i.e.~represents a Dirac measure) and it is maximal for the uniform distribution \citep{Havrda1967}.  It follows that, when $t>1$, the function $\Mag(t\mathcal{M})$ is lower bounded by $\#(T(\bot))$, and this value is attained by a ``deterministic'' language model that has a preferred continuation for each prompt, with no uncertainty (see Example \ref{ex:deterministic}). Similarly, also when $t>1$, the function $\Mag(t\mathcal{M})$ is upper bounded by 
\[\frac{(t-1)(n^{1-t}-1)}{1-t}\cdot \#(\text{ob}(\mathcal{M})\setminus T(\bot)) + \#(T(\bot))\]
and this value is attained by a completely random language model whose next-token probability distributions are always uniform over the set of all possible continuations. This means that for every string $x$, the probability distribution $p_x$ is uniform on the token set; in this case the $t$-logarithmic entropy $H_t(p_x)$ is equal to $(n^{1-t}-1)/(1-t)$, where $n$ is the size of the token set. In the limit $t\to \infty$, these bounds take the simple form
\[\underbrace{\#(T(\bot))}_{\text{deterministic LM}}  \;\leq \; \lim_{t\to \infty} \Mag(t\mathcal{M}) \; \leq \; \underbrace{\#(\text{ob}(\mathcal{M}))}_{\text{maximally random LM}}.\] 

More generally, in the case $t>1$ the magnitude is always proportional to the amount of uncertainty in the model, as measured by the $t$-logarithmic entropy. In fact, if two language models with associated spaces $\mathcal M_1$ and $\mathcal M_2$ assign the same next-token probabilities (respectively, $p^{(1)}_\bullet$ and $p^{(2)}_\bullet$)  to all prompts but one, called $x'$, in such a way that $H_t(p^{(1)}_{x'}) > H_{t}(p^{(2)}_{x})$, then $\Mag(t\mathcal M_1)> \Mag(t\mathcal M_2)$. 
 
The situation is less clear when $t<1$. For instance, for a maximally random LM, we have $\lim_{t\to 0}\Mag(t\mathcal{M})\ = (1-n) \cdot \#(\text{ob}(\mathcal{M})\setminus T(\bot)) + \#(T(\bot)),$ which is negative (since $n>1$, unless $\dagger$ is the only token besides $\bot$). However, even for metric spaces, this limit might take any value \citep{roff2025small}, thus resisting a straightforward interpretation.

We shall see that the derivative of $\Mag(t\mathcal M)$ at $t=1$ is a sum of Shannon entropies: $\sum_{x\in\ob(\mathcal{M})\setminus T(\bot)}H(p_x)$. Analogous results hold for the subspace $\mathcal M_x$ of texts that extend a given prompt $x$, see Remark \ref{rmk:subspace_Mx}.

Let us also remark that dependence of magnitude on the terminating states is reminiscent of similar results pertaining to posets or metric spaces that show a dependency on the ``boundary.'' For instance, if a finite poset has top and bottom elements, then its magnitude can be computed by evaluating its M\"obius function (defined in Section \ref{sec:magnitude}) at those extremal elements \citep{rota1964}, \cite[Proposition 3.8.5]{stanley2011}. And in the context of metric spaces, \emph{weighting vectors} \citep{Leinster2013}|such as $(\sum_y \mu(x,y))_{x}$  when the M\"obius coefficients  $\mu$ exist|effectively detect the boundary of certain subsets of Euclidean space, an insight with recent applications in machine learning \citep{Bunch2021,adamer2024}; see also \citep{willerton2009}. 

Finally, we will draw a connection to magnitude homology \citep{LeinsterShulman2021}. This is our final main result:\\

\noindent \textbf{Magnitude homology associated with an LM} (see Proposition \ref{prop:expression_magnitude}). \quad \textit{The magnitude function can be expressed in terms of the ranks of the magnitude homology groups. For any $t>0,$
\[ \Mag(t\mathcal M) = \sum_{\ell } e^{-t\ell} 
\sum_{k\geq 0} (-1)^k \operatorname{rank}(H_{k,\ell}(\mathcal M)),
\]
where the first sum ranges over those $\ell$ that may appear as finite lengths of paths in $\mathcal M$. }\\

\noindent This result mirrors an analogous result by 
\citet[Theorem 7.14]{LeinsterShulman2021}, although our proof follows directly from Proposition \ref{prop:vigneaux} and encounters significantly less technical complications.

\subsection*{Structure of the article} 
 Section \ref{sec:definingpi} derives, from a given LM, an enriched category of strings and an associated generalized metric space.  Subsection \ref{ssec:tokens} introduces the basic definitions concerning tokens and texts, including the partial order of texts. Subsection \ref{ssec:induced_probs} describes the process of text generation by an LM, defines the function $\pi$, and shows that $\pi(-|x)$ is a probability mass function over the set $T(x)$ of terminating states. Subsection \ref{ssec:catprelims} reminds the reader of some basic definitions from enriched category theory, focusing on categories enriched over commutative monoidal preorders, which are a particularly simple case.  Subsection \ref{ssec:cat_enriched_interval} shows that $\pi$ defines a $[0,1]$-category $\mathcal L$ of texts in the sense of  \citep{ECTL2022}.  Subsection \ref{ssec:cat_enriched_reals} introduces the corresponding generalized metric space $\mathcal M$ using the isomorphism of categories $-\ln:[0,1]\to[0,\infty]$. 
 
 Section \ref{sec:magnitude} is dedicated to magnitude. Subsection \ref{sec:def_magnitude} gives a short overview of categorical magnitude, and Subsection \ref{sec:mag_posets} treats the classical case of posets.  Subsection \ref{ssec:magnitude_main} deals with the computation of the M\"obius coefficients and magnitude of $\mathcal M$.   Subsection \ref{ssec:homology} covers the magnitude homology of $\mathcal M$. Finally, Section \ref{sec:summary} presents some final remarks and perspectives.

\section{Obtaining enriched categories from an LM}
\label{sec:definingpi}

The goal of this section is to show that every LM defines a category enriched over the unit interval that, in turn, gives rise to a generalized metric space. In leading up to these results, we begin with a thin category whose objects are strings of symbols from a finite alphabet that start with a beginning-of-sentence token and that may or may not end with an end-of-sentence token. Morphisms are provided by substring containment. We then enrich this setting by constructing $[0,1]$-hom objects obtained from the probabilities generated by a given LM and also consider a $[0,\infty]$-enriched analogue.

\subsection{Tokens and texts}
\label{ssec:tokens}
Let $A$ be a finite alphabet (a set of tokens) and consider all
finite strings $a,b,\ldots$ from the free monoid $A^*$. In LMs, tokens can correspond to words, pieces of words, or special characters, depending on the model, and hence longer texts (e.g.~sentences, paragraphs) are elements of $A^*$. Below, we call elements of $A$  \emph{tokens} and elements of $A^*$ \emph{strings} or \emph{texts}.  Write $a\leq b$ whenever $a$ is a \emph{prefix} of $b$, that is, whenever $b=aa'$ for some $a'\in A^*$. This defines a partial order since $\leq$ is reflexive (as $A^*$ contains the empty string $\epsilon$), transitive, and antisymmetric. We denote by $|a|$ the length of a string $a\in A^*$. 

Besides the elements of $A$, we introduce two special tokens $\bot$ and $\dagger$ which represent, respectively, the so-called ``beginning-of-sentence'' and ``end-of-sentence'' tokens. Unlike the elements in $A$, the tokens $\bot$ and $\dagger$ do not appear in arbitrary positions; their role is clarified below. Although some systems conflate both symbols, it is useful to keep the distinction here. 

We regard a partially ordered set  $(P,\leq)$ as a thin category with object set $P$ and such that $x\to y$ if and only if $x\leq y$.  Given $N\in \mathbb N$, we define $\mathsf{L}:= \mathsf{L}^{\leq N}$ as the full subcategory of $((A\cup\{\bot, \dagger\})^* ,\leq)$ made of strings that start with $\bot$ and are followed by at most $N-1$ symbols, all of them in $A$ with the exception of the last one, which might equal $\dagger$; in other terms,
\[\ob(\mathsf{L})=\{\bot a : a\in A^* \text{ and } |a|\leq N-1\}\sqcup \{\bot a \dagger : a\in A^* \text{ and } |a|<N-1\}.\]
More explicitly, given objects $x,y$ of $\mathsf{L}$, there is a morphism $x \to y$ whenever $x\leq y$ as elements of $(A\cup\{\bot, \dagger\})^*$, that is, if $x$ is a prefix of $y$.  
We refer to an object of the form $\bot a$ as an \emph{unfinished text} and to an object of the form $\bot a \dagger$ as a \emph{finished text}. 
For any non-identity morphism $x\to y$, the prefix $x$ is necessarily an unfinished text, whereas $y$ can be finished or unfinished. 
Observe that $\mathsf{L}$ has an initial object,\footnote{By this point, the reader may have wondered why $\dagger$ is used as the end-of-sentence symbol instead of $\top$, given our use of $\bot$ for the beginning-of-sentence symbol. To start, the category $\mathsf{L}$ does not have a terminal object, so we wish to avoid using notation that suggests otherwise. Further, one might think of the end-of-sentence token as the state which ``kills'' the generation of a string, hence a dagger.} which is $\bot$, and also that it has a natural ``grading'' given by the length $|x|$ of strings $x$. We remark that $|\bot|=1$. 

It will be convenient to write the set $\ob(\mathsf{L})$ as a disjoint union $\bigsqcup_{j= 0}^{N-1}\mathsf{L}^{(j)}$, where $\mathsf L^{(j)}$ consists of strings of length $1+j$. (Remark that $\mathsf L$ is a free category on a graph, so $j$ measures the graph-theoretic distance of the objects in $L^{(j)}$ to $\bot$.)  Aside from identities, the arrows of $\mathsf L$ only go from elements of $\mathsf{L}^{(i)}$ to elements of $\mathsf{L}^{(j)}$ for $j>i$. 
Similarly, given a string $x$ in $\mathsf L$, there is a full subcategory $\mathsf{L}_x$ of $\mathsf{L}$ whose objects are $y\in \mathsf L$ such that $x\to y$; in this case, we have $\ob(\mathsf L_x) = \bigsqcup_{j= 0}^{N-|x|} \mathsf{L}_x^{(j)}$ , where  $\mathsf{L}_x^{(j)}$ consists only of those strings of length $|x|+j$, namely, those strings that extend $x$ on the right by $j$ tokens. So, $\mathsf{L}_x^{(0)}=\{x\}$ and  $\mathsf{L}_x^{(1)}=\{xa_1:a_1\in A\cup\{\dagger\}\}$ and $\mathsf{L}_x^{(2)}=\{xa_1a_2:a_1\in A, a_2\in A\cup\{\dagger\}\}$ and so on. Remark that $\mathsf L_\bot^{(j)} = \mathsf L^{(j)}$ for any $j\geq 0$, and if $x$ is a finished text, then $\mathsf{L}_x$ has a single object, namely $x.$ 

We will refer to $\mathsf{L}$ as a poset or as a category interchangeably. Either way, it keeps track of which strings are right extensions of other strings in a language. To incorporate statistical information, we turn to enriched category theory in the following sections.

\subsection{LMs and induced probabilities on texts}
\label{ssec:induced_probs}
We characterize the behavior of an autoregressive language model   (e.g.~GPT, Llama) as follows. For any tokenized user input $a\in A^*$,  the model is fed the  prompt $x=\bot a$ and generates a probability distribution $p_x:=p(-|x)\colon A\cup \{\dagger\}\to[0,1]$.    
After generating $p_x$, the model then samples a token  $a_1$ according to $p_x$. If $a_1 = \dagger$, then the process terminates and the model outputs $a$. Otherwise, the token $a_1$ is appended to $a$. These steps are then repeated. The model generates a probability distribution $p_{\bot aa_1}$ on $A\cup \{\dagger\}$, samples a token $a_2$ according to it, and either terminates execution if $a_2 = \dagger$ (outputting $aa_1$) or appends $a_2$ to $aa_1$ to generate $p_{\bot aa_1a_2}$, and so on.

We assume that the process stops when $\dagger$ is sampled or when the extended prompt $y = xa_1 \cdots a_{N-|x|}$ reaches a maximum length $N$; we call $N$ the  \emph{cutoff}. We can think of $N-1$ as the context size of the language model, i.e.~the maximum length of a string $x$ that can be used to generate a distribution $p_x$. Below are some common examples of autoregressive language models.

\begin{example}\label{ex:deterministic}
We say that a language model is \emph{deterministic} if, for any possible prompt $x$, there is $a=a(x)\in A\cup\{\dagger\}$ such that $p_x(a)=1$ (and therefore $p_x(a')=0$ for any $a'\in (A\cup\{\dagger\})\setminus\{a\}$). Although this example is quite artificial, it will appear as an extreme case for our magnitude calculations. 
\end{example}

\begin{example} \label{ex:markov_chain}
If $p_{\bot a_1 \cdots a_n}(-)$ only depends on $a_n$ for any $n\geq 1$, then the LM corresponds to a \emph{Markov chain}. In the terminology of Markov chains (see e.g.~\cite[Ch.~6]{Grimmett2020}), $\tilde A = A\cup \{\dagger\}$ is the set of states and $\dagger$ is an absorbing state. The numbers $P_n(a_n, a_{n+1}) := p_{\bot a_1 \cdots a_n}(a_{n+1})$ form a matrix $P_n:\tilde A\times \tilde A \to [0,1]$, which has positive entries and is \emph{stochastic}, meaning that each row sums to $1$, i.e.~$\sum_{a'\in \tilde A} P_n(a,a')=1$. By convention we are assuming here that $P_n(\dagger, \dagger) = 1$ and $P_n(\dagger, a) = 0$ for any $a\in A$, which makes $\dagger$ an absorbing state.  

Later in Section \ref{sec:summary}, we will briefly remark on a notion called \emph{categorical diversity} in connection with homogeneous and irreducible Markov chains. If for all $n>2$ it holds that $P_n = P_1$, then the chain is said to be \emph{homogeneous}. A Markov chain is \emph{irreducible} if there is a nontrivial probability of going from any state $a\in A$ to any state $a'\in A$ in a finite number of steps. A chain with an absorbing state is never irreducible. However, if $P_n(a,\dagger)=0$ for every $n\geq 1$ and $a\in A$, then $P_n|_{A\times A}$ is also stochastic and $(P_n)_{n\geq 1}$ defines a Markov chain with state space $A$. When this resulting chain is homogeneous and irreducible, it necessarily has a stationary distribution (because $A$ is finite, see \cite[Sec. 6.6]{Grimmett2020}), which is a probability vector $q:A\to [0,1]$ such that $\sum_{a'\in A} q(a) P(a,a') = q(a')$. In this setting, we will see in Section \ref{sec:summary} that if additionally $p_{\bot}(a) = q(a)$, then we can recover the Kolmogorov--Sinai entropy of the Markov chain from the categorical diversity. 
\end{example}

\begin{example}
An \emph{$n$-gram model} is a Markov chain ``with memory'' of order $n-1$, in the sense that $p_{\bot a_1 \cdots a_i} (-)$ only depends on the last $n-1$ tokens of $\bot a_1 \cdots a_i$ (or all of them if $i< n-2$). For details see \cite[Ch.~3]{Jurafsky2025}.  
\end{example}

\begin{example}
    A \emph{large language model} is a deep neural network that implements a function $p^{(\theta)}:A^* \to \Delta(A), \quad a\mapsto p^{(\theta)}_{\bot a} (-)$. Here $\Delta(A)$ denotes the set of probability mass functions on $A$, and $\theta$  the network's parameters (connection weights between neurons) and  hyper-parameters. In the case of autoregressive LLMs, the deep neural network usually implements a version of the decoder-only transformer architecture \citep{Vaswani2017, Radford2019}. The network is trained on a self-supervised manner: the goal is to predict as well as possible the last token $a_n$ from $(a_1,...,a_{n-1}$), given any tokenized fragment of text $(a_1,...,a_n)$ in a corpus $C\subset A^*$. For details see \cite[Ch.~10]{Jurafsky2025}. 
\end{example}

In general, it will be useful to consider the set of all possible outputs of an LM corresponding to a given input. In this vein, we define the set of \emph{terminating states} for a prompt $x=\bot a$ as follows.

\begin{definition}\label{def:term_states}
The set of \emph{terminating states} of an unfinished text $x$ in $\mathsf{L}$ is defined to be
\begin{multline*}
    T(x) = \{ y \in \ob(\mathsf L_x) : 
    y\text{ is an unfinished text of length }N \text{ or }\\ y \text{ is a finished text such that } |y|\leq N\}. 
\end{multline*}
\end{definition}
In the first case, $y=xa'$ for some $a'\in A^*$ such that $|a'|=N-|x|$; 
for such a terminating state the model outputs $aa'$ to the user according to some probability distribution. In the second case, $y=xa''\dagger$ for some $a''\in A^{*}$ such that $|a''|\leq N-|x|-1$; for such a terminating state the model outputs $aa''$ according to some probability distribution. Either way, there is a bijective correspondence between the terminating states and the theoretically possible outputs of the model, including those with small, or even zero, probabilities.\footnote{Notice that this description is only valid for prompts whose length is strictly less than $N$:  to extend a prompt $x$ of length greater than $N-1$, one needs to produce first a subrogate prompt $f(x)$ such that $|f(x)| \leq N-1$ and feed it to the language model to generate $p_{f(x)}$ and sample a token.} The phrase ``theoretically possible''  is used to emphasize that the set $T(x)$ depends only on the category of strings $\mathsf{L}$ and is independent of a choice of language model. 

Now, it is tempting to define the probability of the model's output $b$ given the user input $a$  as the product of the intermediate probabilities of the tokens involved in its production. For example, when $N>5$, the probability of generating the output $b=a_1a_2a_3a_4a_5\in A^*$ given the user's input $a=a_1a_2$ would be
\[p(a_3|\bot a)p(a_4|\bot a a_3)p(a_5|\bot a a_3 a_4 )p(\dagger|\bot a a_3a_4a_5).\] However, it is not obvious that this rule indeed defines a probability mass function over some set.  We prove this below. But first, let us define these products in general.  

\begin{definition}\label{def:pi}
 For any objects $x$ and $y$ of $\mathsf{L}$, define
\begin{equation}\label{eq:pi}
\pi(y|x):=
\begin{cases}
    1 & \text{if } x=y\\[5pt]
    0 & \text{if } x\not\to y\\[5pt]
    \displaystyle\prod_{i=1}^{k}p(a_{t+i}|y_{<t+i}) & \text{if } x\to y\\[5pt]
\end{cases}.
\end{equation}
In the third case, we assume that $x= \bot a_1\cdots a_t$ and $y =xa_{t+1}\cdots a_{t+k}$ for some $k\geq 1$, $(a_i)_{i=1}^{t+k-1} \subset A$, and  $a_{t+k} \in A\cup\{\dagger\}$. The symbol $y_{<t+i}$ denotes the string $\bot a_1\cdots a_{t+i-1}$, and $y_{<t+1}=x$. 
\end{definition}

For a given string $x$ in $\mathsf{L}$, the function $\pi(-|x)$ is not a probability mass function on its whole support. (To begin with, $\pi(x|x)=1$.) Nonetheless, it becomes a probability mass function when restricted to $T(x)$.

\begin{proposition}\label{prop:prob_terminal}
    Every autoregressive language model determines a probability mass function $\pi(-|x)$ on the set $T(x)$ of terminating states of an unfinished text $x$. 
\end{proposition}
\begin{proof}
    If $m=0$, then $T(x) = \{x\}$ and by definition $\pi(x|x)=1$. If $m=1$, then $T(x)=\{xa\mid a\in A\cup \{\dagger\}\}$ and 
    $$\sum_{y\in T(x)} \pi(y|x) = \sum_{a\in A\cup \{\dagger\}} p_x(a) =1,$$
    since $p_x$ is assumed to be a probability mass function on $A\cup\{\dagger\}$.
For general $m\geq 1$, 
    \begin{align}
        \sum_{y\in T(x)} \pi(y|x) &= \sum_{i=0}^{m-1} \sum_{a'\in A^i} \pi(xa'\dagger |x) + \sum_{a\in A^{m}} \pi(xa|x) \label{eq:decomp_T} \\
       &= \sum_{i=0}^{m-1} \sum_{a'\in A^i} \pi(xa'\dagger |x)  +  \sum_{\substack{ a=a'a''\\a'\in A^{m-1},\, a''\in A}} p(a''|xa')\pi(xa'|x)  \label{eq:decomp_a}  \\
       &= \sum_{i=0}^{m-2} \sum_{a'\in A^i} \pi(xa'\dagger |x) + \sum_{a'\in A^{m-1}} \pi(xa'|x) \sum_{a''\in A\cup \{\dagger\}} p(a''|xa') \label{eq:reordering} \\
       &= \sum_{i=0}^{m-2} \sum_{a'\in A^i} \pi(xa'\dagger |x) + \sum_{a'\in A^{m-1}} \pi(xa'|x) \label{eq:previous_step}
    \end{align}
   The two sums in \eqref{eq:decomp_T} follow from the definition of $T(x)$ and $\pi$, where the first sum accounts for finished texts (i.e.~those ending in $\dagger)$, and the second sum accounts for unfinished texts of length $N$. We obtain \eqref{eq:decomp_a} by identifying $a\in A^m$ with $(a',a'')\in A^{m-1}\times A$. Then \eqref{eq:reordering} follows by rewriting the term corresponding to $i=m-1$ in the leftmost sum as $\pi(xa'|x)p(\dagger|xa')$, where $a'\in A^{m-1}$, and then moving it to the rightmost sum. The expression in \eqref{eq:previous_step} equals one in virtue of the induction hypothesis. 
\end{proof}

\subsection{Preliminaries on enriched category theory}
\label{ssec:catprelims}

We shall now use $\pi$ to derive an enriched category of strings $\mathcal L$ from $\mathsf L$. For the convenience of the reader, we give here the definition of a category enriched over a commutative monoidal preorder \cite[Chapter 2]{fong2019invitation}, which is the only case we will need. See also  \citep{Kelly82}. 

\begin{definition}
A \emph{commutative monoidal preorder} $(\mathcal{V},\leq,\otimes,1)$ is a preordered set $(\mathcal{V},\leq)$ and a commutative monoid $(\mathcal{V},\otimes,1)$ satisfying $x\otimes y\leq x'\otimes y'$ whenever $x\leq x'$ and $y\leq y'.$
\end{definition}

\begin{example} The unit interval $([0,1],\leq,\cdot,1)$ is a commutative monoidal preorder with the usual ordering $\leq$. Multiplication of real numbers is the monoidal product, which we will denote by juxtaposition, $ab:=a\cdot b$ for $a,b\in [0,1]$, and the monoidal unit is 1. 
\end{example}

\begin{example}
The extended non-negative reals $([0,\infty],\geq,+,0)$ form a commutative monoidal preorder where the preorder is the opposite of the usual ordering on the reals. The monoidal product is addition with $a+\infty:=\infty$ and $\infty+a:=\infty$ for all $a\in[0,\infty]$, and the monoidal unit is 0. 
\end{example}

\begin{definition}\label{def:enrichedcat}
    Let $(\mathcal{V},\leq,\otimes,1)$ be a commutative monoidal preorder. A (small) \emph{category enriched over $\mathcal{V}$}, or simply a \emph{$\mathcal{V}$-category}, $\mathcal{C}$ consists of a set $\ob(\mathcal{C})$ of objects and, for every pair of objects $x$ and $y$, an object $\mathcal{C}(x,y)$ of $\mathcal{V}$ called a \emph{$\mathcal{V}$-hom object} satisfying the following: for all objects $x,y,z\in\ob(\mathcal{C})$,
    \begin{align*}
    1\leq \mathcal{C}(x,x) \\
    \mathcal{C}(y,z)\otimes\mathcal{C}(x,y)\leq \mathcal{C}(x,z)
    \end{align*}
\end{definition}

Sections \ref{ssec:cat_enriched_interval} and \ref{ssec:cat_enriched_reals} below give examples of categories enriched over both $[0,1]$ and $[0,\infty]$ defined from the probabilities generated by a  language model. As we will see in the setup introduced in Subsections \ref{ssec:tokens} and \ref{ssec:induced_probs}, objects of both categories will consist of strings of characters from some finite token set.

\subsection{Every LM defines a $[0,1]$-category}\label{ssec:cat_enriched_interval}
For convenience, let us recall the setup established in Subsections \ref{ssec:tokens} and \ref{ssec:induced_probs}. As described there, $A$ denotes a finite alphabet, such as the token set of a given LM, and $\mathsf{L}$ denotes the finite category whose objects are strings from $A$ that begin with $\bot$ and are followed by at most $N-1$ symbols, where the last symbol may or may not be $\dagger$. Here $N$ denotes the model's cutoff size. Moreover, there is a morphism $x\to y$ if $x$ is a prefix of $y$. Further recall that for each pair of strings $x,y$ in $\mathsf{L}$, Definition \ref{def:pi} specifies the value $\pi(y|x)\in[0,1]$ which, if $y$ is an extension of $x$, is equal to the product of the successive probabilities given by the model when generating $y$ from $x$ one token at a time. Otherwise, if $x$ is not contained in $y$, then $\pi(y|x)=0$, and if $x=y$ then $\pi(x|x)=1$.

This allows us to now define a category $\mathcal{L}$ enriched over the unit interval. The objects of $\mathcal L$ coincide with the objects of $\mathsf L$, and for a pair of objects $x$ and $y$ we define their $[0,1]$-hom object to be $\mathcal{L}(x,y):=\pi(y|x).$  

Let us verify that $\mathcal{L}$ satisfies both the identity and compositionality requirements for a category enriched over $[0,1]$ listed  in Definition \ref{def:enrichedcat}. Suppose first that $x,y,z$ are three strings satisfying $x\to y\to z$.
If $x\neq y$ and $y\neq z$, then $x$ and $y$ are necessarily unfinished texts, and we may write
\begin{align*}
x&= \bot a_1\cdots a_t\\
y&=\bot a_1\cdots a_t\cdots a_{t+k}\\
z&=\bot a_1\cdots a_t\cdots a_{t+k}\cdots a_{t+k+k'}.
\end{align*}
for some $(a_i)_{i=1}^{t+k+k'-1}\subset A$ and $a_{t+k+k'}\in A\cup\{\dagger\}$. It then follows from  \eqref{eq:pi} that
\begin{align}\label{eq:compositionality}
\pi(y|x)\pi(z|y)&=\prod_{i=1}^{k}p(a_{t+i}|y_{<t+i})\prod_{j=1}^{k'}p(a_{t+k+j}|z_{<t+k+j})\\[5pt]  \nonumber 
&=\prod_{i=1}^{k+k'}p(a_{t+i}|z_{<t+i})\\[5pt]   \nonumber 
&=\pi(z|x).
\end{align}
If instead $x=y$, then $\pi(y|x)=1$ and obviously $\pi(z|y) = \pi(z|x)$. The case $z=y$ is treated similarly. More generally, if $x,y$ and $z$ are arbitrary objects of $\mathsf L$, one has that $\pi(y|x)\pi(z|y)\leq \pi(z|x)$ as it may be the case that $\pi(y|x),\pi(z|x)\neq 0$ but $\pi(z|y)=0$,\footnote{For example, consider when $x=\bot\texttt{green}$ and $y=\bot\texttt{green lantern}$ and $z=\bot\texttt{green salad}$.} or that $\pi(z|y),\pi(z|x)\neq 0$ but $\pi(y|x)=0$.\footnote{For example, consider when $x=\bot\texttt{movie tic}$ and $y=\bot\texttt{movie}$ and $z=\bot\texttt{movie ticket}$.} In summary, we come to the following proposition.

\begin{proposition}\label{prop:LLM_defines_ecat}
Every autoregressive language model defines a $[0,1]$-category whose objects are those of $\mathsf L$ and whose hom-objects are given by $\pi$.
\end{proposition}

A more general version of this $[0,1]$-category was originally defined in \cite[Definition 4]{ECTL2022}, where  $\mathcal{L}(x,y)$ was taken to be nonzero whenever $x$ was an \textit{arbitrary} substring of $y$ (not strictly a prefix). However, the values $\pi(y|x)$ were not constructed explicitly from an LM, whereas Definition \ref{def:pi} above yields a $[0,1]$-category of expressions in a language where the values $\pi(y|x)$ are described concretely from the probabilities generated by an LM. A similar definition can be found in \citep{GV2024}, although neither their description nor the one in \citep{ECTL2022} considers the special characters $\bot$ and $\dagger$ or a model's cutoff value. The incorporation of these aspects led us to the novel Proposition \ref{prop:prob_terminal}, which \emph{justifies referring to $\pi$ as a probability.}

Notice we only consider extensions of expressions \textit{on the right,} having in mind the most standard, autoregressive LLMs, although one could also consider bidirectional extensions. That said, the tree-like structure implied by the restriction to right extensions plays a key role in the computation of magnitude presented below.

Finally, let us also remark that, after defining the $[0,1]$-category $\mathcal{L}$ of strings, the authors of \citep{ECTL2022} further consider enriched copresheaves on $\mathcal{L}$, which were shown to contain semantic information. Our present goal, however, is to turn our attention away from $\mathcal{L}$ to a more geometric version of it.

\subsection{Every LM defines a $[0,\infty]$-category}
\label{ssec:cat_enriched_reals}
As discussed in \cite[Section 5]{ECTL2022}, the function $-\ln\colon [0,1]\to[0,\infty]$ is an isomorphism of categories, which provides a passage from probabilities to a more geometric setting. That is, instead of considering an enrichment from probabilities involved in generating a string $y$ from a prefix $x$, one could instead think of the ``distance'' traveled from $x$ to $y$ by defining
\[d(x,y):=-\ln\pi(y|x).\]
It is straightforward to check the triangle inequality is satisfied $d(x,y)+d(y,z)\geq d(x,z)$ and further that $d(x,x)=0$ for all strings $x,y,z.$ (This follows from Equation \eqref{eq:compositionality} and the fact that $\pi(x|x)=1$ for all $x$.) From this perspective, a text that is highly likely to extend a prompt $x$ is close to $x$, and a text that is not an extension of $x$ is infinitely far way. In this way, we obtain a category $\mathcal{M}$ enriched over $[0,\infty]$, also known as a \emph{generalized metric space} in the sense of 
\citet{lawvere73}, by defining the $[0,\infty]$-hom object between a pair of strings $x,y$ to be the distance $\mathcal{M}(x,y):=d(x,y)$. And now that we have a generalized metric space, we may inquire after its magnitude.

\section{Magnitude}
\label{sec:magnitude}

The theory behind the magnitude of (generalized) metric spaces is well known \citep{Leinster2013} and is a special case of the magnitude of enriched categories \citep{LeinsterShulman2021,LeinsterMeckes}. Importantly, the theory requires working with a \textit{finite} enriched category, which is an additional reason to consider the category $\mathsf L^{\leq N}$ with finite cutoff $N$. In this section, then, we provide a few preliminaries before computing the magnitude of the generalized metric space $\mathcal{M}$. We begin by reviewing the definition of magnitude in the context of enriched categories and in the classical context of posets. 

\subsection{Definition}\label{sec:def_magnitude}
Magnitude is a numerical invariant of a category enriched over a monoidal category $(\mathcal{V},\otimes, 1)$ \citep{Leinster2013,LeinsterShulman2021,LeinsterMeckes}. Briefly, one starts with a semi-ring $R$ and a multiplicative
function $\lVert- \rVert\colon \ob(\mathcal{V})\to R$ called \emph{size} that is invariant under isomorphisms. So, $\lVert v \rVert=\lVert w \rVert$ whenever $v\cong w$ in $\mathcal{V}$, and $\lVert 1 \rVert=1$ and $\lVert v\otimes w \rVert=\lVert v \rVert \lVert w  \rVert$ for all $v,w\in\ob(\mathcal{V})$. Then, given a $\mathcal{V}$-category $\mathcal{C}$ with finitely many objects, one introduces the \emph{zeta function} $\zeta_\mathcal{C}\colon \ob(\mathcal{C})\times\ob(\mathcal{C})\to R$, defined by 
$\zeta_\mathcal{C}(x,y):=\lVert\mathcal{C}(x,y) \rVert$. The function $\zeta_\mathcal{C}$ can be regarded as a square matrix with index set $\ob(\mathcal C)$ \citep{Leinster2013}; when it is invertible,  $\mathcal{C}$ is said to have \emph{M\"obius inversion} and its \emph{M\"obius function} is $\mu_\mathcal{C}:=\zeta_\mathcal{C}^{-1}$ (the entries of this matrix are called \emph{M\"obius coefficients}). When $\mathcal C$ has M\"obius inversion, the \emph{magnitude} $\Mag(\mathcal{C})$ of $\mathcal{C}$ can be defined as
\begin{equation}\label{eq:magnitude}
\Mag(\mathcal{C})=\sum_{(x,y)\in \ob(\mathcal{C})\times\ob(\mathcal{C})}\zeta_\mathcal{C}^{-1}(x,y).
\end{equation}

For $\mathcal{V}=[0,\infty]$, every size function is of the form $\lVert x \rVert_t = e^{-tx}$ for some $t\in \mathbb R \cup\{\infty\}$. It is customary to choose $\lVert  - \rVert_1$ and then consider the \emph{ magnitude function} $f(t):=\Mag(t\mathcal{C})$,  defined for $t\in (0,\infty)$ \cite[Section 2.2]{Leinster2013}.  The symbol $t\mathcal{C}$ denotes the generalized metric space with the same objects as $\mathcal{C}$ but whose distances (that is, whose $[0,\infty]$-hom objects) are scaled by $t$.

\subsection{Magnitude of posets}\label{sec:mag_posets} The magnitude of an enriched category stems from a classical story from posets, which we now briefly review from \citep{stanley2011}. 

Given a poset $P$, define a \emph{closed interval} to be $[s,t]=\{u\in P:s\leq u \leq t\}$ whenever $s\leq t$ in $P$.
If every interval of $P$ is finite, then $P$ is said to be a \emph{locally finite} poset. 

Let $\mathbb{k}$ be a field, $P$  a locally finite poset,  and  $\Int(P)$ the set of all closed intervals of $P$.  The \emph{incidence algebra} $I(P,\mathbb{k})$ of $P$ over $\mathbb{k}$ is the $\mathbb{k}$-algebra of all functions $f\colon \Int(P)\to\mathbb{k}$ with the usual structure of a vector space over $\mathbb{k}$, where multiplication is given by the \emph{convolution product}, 
\[(f\ast g)(s,u):=\sum_{s\leq t \leq u}f(s,t)g(t,u),\]
where we write $f(s,t)$ for $f([s,t])$. (Remark that the sum has finitely many terms.) The incidence algebra of $P$ is an associative algebra with identity $\delta\colon\Int(P)~\to~\mathbb{k}$ given by
\[
\delta(s,u)=
\begin{cases}
    1 &\text{if }s=u\\
    0 &\text{if }s\neq u
\end{cases}
\]
since $f\ast \delta=\delta \ast f=f$ for all functions $f\in I(P,\mathbb{k})$. The  \emph{zeta function} $\zeta_P$ is another notable function in the incidence algebra, defined to be constant at $1$ on every interval,
\[\zeta_P(s,u)=1, \quad \text{for all } s\leq u \text{ in } P.\]
Observe that this coincides with the zeta function defined in the enriched categorical setting, since every poset $P$ (and in particular the poset $\mathsf{L}$ of strings) may be viewed as a category enriched over truth values $\mathbf 2:=\{\texttt{true, false}\}.$ Its objects are the elements of $P$, and given $s,u\in P$, if $s\leq u$, then their corresponding hom object $P(s,u)$ is \texttt{true}, and if $s\not\leq u$ then it is \texttt{false}. The identity and compositionality requirements in Definition \ref{def:enrichedcat} arise from the reflexivity and transitivity of the partial order and by considering the size function $\lVert -\rVert\colon \ob(\mathbf 2)\to\mathbb{Z}$ given by $\lVert \texttt{true}\rVert=1$ and $\lVert\texttt{false}\rVert=0$ and then setting $\zeta_P(s,u)=\lVert P(s,u)\rVert$.

For every locally finite poset $P$, the function $\zeta_P$ is an invertible element of the incidence algebra $I(P,\mathbb{k})$  \cite[Prop. 3.6.2]{stanley2011}. Equivalently, the zeta function is invertible when regarded as a square matrix with index set $P$, cf.~\citep{Leinster2013}; we have taken this matricial perspective in our presentation of the more general categorical definition in Subsection \ref{sec:def_magnitude}. The inverse of $\zeta_P$ is called the \emph{M\"obius function} of $P$, denoted $\mu_P$, and it can be computed recursively: 
\[
\mu_P(s,u) = 
\begin{cases}
1 & \text{if } s=u\\[5pt]
-\displaystyle\sum_{s\leq t < u}\mu_P(s,t) &\text{if } s< u\\[5pt]
0 &\text{otherwise}.
\end{cases}
\]

\begin{example}\label{ex:mobius}
The M\"obius function $\mu_{\mathsf{L}}$ on the poset $\mathsf{L}$ of strings is rather simple. This is due to the tree-like structure of the full subcategory $\mathsf{L}_x$ of $\mathsf{L}$ introduced in Subsection~\ref{ssec:tokens}, whose objects are all strings $y$ having a given string $x$ as a prefix. Indeed, for any string $x\in \ob (\mathsf L)$ and tokens $a_1,a_2,\ldots\in A\cup\{\dagger\}$, we have
\begin{align*}
\mu_{\mathsf{L}}(x,x)&=1\\
\mu_{\mathsf{L}}(x,xa_1)&=-\mu_{\mathsf{L}}(x,x)=-1\\
\mu_{\mathsf{L}}(x,xa_1a_2)&=-\mu_{\mathsf{L}}(x,x)-\mu_{\mathsf{L}}(x,xa_1)=-1+1=0\\
\mu_{\mathsf{L}}(x,xa_1a_2a_3)&=-\mu_{\mathsf{L}}(x,x)-\mu_{\mathsf{L}}(x,xa_1) -\mu_{\mathsf{L}}(x,xa_1a_2)=-1+1+0=0.
\end{align*}
Similarly, $\mu_{\mathsf{L}}(x,xa_1a_2\cdots a_j)=0$ for $j>1$. (As one of the reviewers pointed out, these identities hold for any free category, so in particular for $\mathsf L$, which is a free category on a tree rooted in $\bot$.) 
\end{example}

This M\"obius function may then be used to compute the magnitude of the finite category $\mathsf{L}$ of strings from a set $A$ of tokens.

\begin{example}\label{ex:mag_of_L} Since the category $\mathsf{L}$ has an initial object, we know its magnitude is 1, cf.~\cite[Example 2.3]{Leinster2008}, but it is also instructive to prove this via an explicit computation of the Möbius function.  Although we recommend the reader performs this computation by herself, we also include it here for the sake of completeness.  To that end, let $\#A$ be the cardinality of the token set $A$ and let $\#\ob(\mathsf{L})$ be the number of strings in the finite poset $\mathsf{L}$ considered as a category. The magnitude of $\mathsf{L}$ is  given by
\begin{align*}
    \Mag(\mathsf{L})&=\sum_{x,y\in\ob(\mathsf{L})}\zeta_\mathsf{L}^{-1}(x,y)\\[5pt]
    &=\sum_{x,y\in\ob(\mathsf{L})}\mu_{\mathsf{L}}(x,y)\\[5pt]
    &=\sum_{x\in\ob(\mathsf{L})}\mu_{\mathsf{L}}(x,x) + \sum_{\substack{x=\bot a\in\ob(\mathsf{L})\\[3pt] a\in \bigcup_{i=0}^{N-2} A^i}} \; \sum_{a_1\in A\cup\{\dagger\}}\mu_{\mathsf{L}}(x,xa_1)\\[5pt]
    &=\#\ob(\mathsf{L})- \#\left\{x\in\ob(\mathsf{L}):x=\bot a \text{ with }a\in \bigcup_{i=0}^{N-2} A^i\right\} (\#A+1).
\end{align*}
Unwinding the last line, recall that $\mathsf{L}$ denotes $\mathsf{L}^{\leq N} = \bigsqcup_{i= 0}^{N-1} \mathsf{L}^{(i)}$. We have $\mathsf{L}^{(0)}=\{\bot\}$ and, for each $i\leq 1$,  $\mathsf{L}^{(i)}$ is comprised of strings of the form $\bot a$ with $a\in A^{i}$ or $\bot a'\dagger$ with $a'\in A^{i-1}$. It follows that $ \#\mathsf{L}^{(i)} = \#A^i + \#A^{i-1}$ and so
\begin{align*}
    \#\ob(\mathsf{L}) &= 1 + \sum_{i=1}^{N-1}\#\mathsf{L}^{(i)} \\
    &= 1+ \sum_{i=1}^{N-1} \#A^i + \#A^{i-1} = \#A^{N-1} + 2(\#A^{N-2}) + \cdots +2(\#A) + 2.
\end{align*}
Furthermore, 
\begin{multline*}
    \#\left\{x\in\ob(\mathsf{L}):x=\bot a \text{ with }a\in \bigcup_{i=0}^{N-2} A^i\right\} (\#A+1)= \left( \sum_{i=0}^{N-2} \#A^i\right) (\#A+1) \\= \#A^{N-1} + 2(\#A^{N-2}) + \cdots + 2(\#A) + 1.
\end{multline*}
Therefore, $\Mag(\mathsf{L}) = 1$.
\end{example}

\subsection{Main results}
\label{ssec:magnitude_main}
We are now ready to compute the magnitude of the generalized metric space associated with an LM. As before, let $\mathcal{M}$ be the $[0,\infty]$-category of strings introduced in Subsection \ref{ssec:cat_enriched_reals}, whose objects are $\ob(\mathcal M) = \ob(\mathsf L^{\leq N})$ and where the $[0,\infty]$-hom objects are given by $\mathcal M(x,y) =  d(x,y)= -\ln \pi(y|x)$.   Define the \emph{zeta matrix} $\zeta_\mathcal{M}\colon\ob(\mathcal{M})\times~\ob(\mathcal{M})~\to~\mathbb{R}$ to be
\[\zeta_{\mathcal{M}}(x,y):=e^{-d(x,y)}=\pi(y|x),\]
and more generally for any real number $t>0$ define
\[(\zeta_\mathcal{M})_t(x,y):=e^{-td(x,y)}=\pi(y|x)^t.\]
It is understood that $d(x,y)=\infty$ whenever $\pi(y|x)=0$; for instance, the distance from a finished text $x$ to any other string $y$ is infinite.
For simplicity, we will write $\zeta_t$ instead of $(\zeta_\mathcal{M})_t$. 
As we will see in the proposition and corollaries below, $\zeta_t^{-1}$ exists and  its components can be computed explicitly. 

\begin{proposition}\label{prop:vigneaux}
For any $t>0$, the matrix $\zeta_t$ is invertible. Moreover, for any $x,y$ in $\ob(\mathcal{M}),$
\begin{equation}\label{eq:generalized_hall}
   \zeta_t^{-1}(x,y)= \sum_{k\geq 0}\sum_{\substack{\text{nondeg. paths}\\ x=y_0\to y_1\to\cdots \to y_k=y}\text{ in } \mathsf L}(-1)^k\prod_{i=1}^k\pi(y_i|y_{i-1})^t.
\end{equation}
\end{proposition}
\noindent(A \emph{non-degenerate path} $x=y_0\to y_1\to\cdots \to y_k=y$ consists of a sequence of composable non-identity morphisms.)
\begin{proof} 
Following \citet{rota1964} (see also \citep{LeinsterShulman2021}), we introduce the formal expansion
\begin{equation}\label{eq:infinite_expansion}
    \zeta_t^{-1} = \sum_{k\geq 0} (-1)^k (\zeta_t-\delta)^k,
\end{equation}
where $\delta$ is the identity matrix: $\delta(x,y) = 1$ if $x=y$ and $\delta(x,y) = 0$ otherwise. 

The matrix $\zeta-\delta$ is the adjacency matrix of a directed graph $D$ with vertices $\ob(\mathcal M)$ (equivalently, $\ob(\mathsf L)$) and edges $E=\{(z,z') \mid z \to z' \text{ in } \mathsf{L},\, z\neq z'\}$ (we are excluding edges coming from identity morphisms, whose weights vanish).  The power $(\zeta_t-\delta)^k$ counts non-degenerate paths of length $k$ in $D$. Since $\mathsf L$ is a finite poset, one readily verifies that $E$ is finite and that the graph has no cycles. Hence the sum in \eqref{eq:infinite_expansion} has a finite number of terms.\footnote{For general categories, one has to establish convergence of the series under certain conditions, see \citep{LeinsterShulman2021}.} 

The expression \eqref{eq:generalized_hall}  arises from evaluating \eqref{eq:infinite_expansion} on $(x,y)$. In fact,
\begin{equation*}
    (\zeta_t - \delta)(x,y) = 
        \pi(x,y)^t \mathbf 1_{(x,y)\in E},
\end{equation*}
where $\mathbf 1_\bullet$ denotes the indicator function, and more generally 
\begin{equation*}
    (\zeta_t - \delta)^k(x,y) = \sum_{\substack{(y_0,y_1,...,y_{k-1},y_k)\\ y_0=x,\, y_k = y \\ (y_{i-1}, y_{i})\in E \text{ for } i=1,...,k}} \prod_{i=1}^k\pi(y_i|y_{i-1})^t.
\end{equation*}
\end{proof}

\begin{remark}
    We envision extensions of this theory where the category $\mathsf L$ would not be a poset. This would be the case, for instance, if we consider the general relation of being a substring, instead of being just a prefix, cf.~\citep{GV2024}. Similarly, it might be possible to build an analogous category for syntactic trees, where morphisms would arise from the syntactic Merge operation and rearrangements of the subtrees of a tree followed by Merge. (Merge is the fundamental syntactic operation in Chomsky's minimalist program; it has been formalized as an operation on a Hopf algebra of syntactic forests in \citep{marcolli2025mathematical}.) In these situations, where the categories present nontrivial cycles, we would need different techniques to compute $\zeta_t^{-1}$, such as those introduced in \citep{vigneaux2024} by the second author. Our Proposition \ref{prop:vigneaux} can also be derived from this more general formalism.  
\end{remark}

The following corollary shows that the M\"obius coefficients obtained from $\zeta_t^{-1}$ have a simplified expression in terms of $\pi$ from Definition \ref{def:pi} and the M\"obius function $\zeta_\mathsf{L}^{-1}$, whose zeta function $\zeta_\mathsf{L}$ was described in Section~\ref{sec:mag_posets}. 

\begin{corollary}\label{cor:zeta_inverse1} 
For any strings $x,y$ in $\ob(\mathcal{M})$,
\[\zeta_t^{-1}(x,y)=\pi(y|x)^t\;\zeta_\mathsf{L}^{-1}(x,y).\]
\end{corollary}
\begin{proof} 
For any $k\geq 0$ and any nondegenerate path $x=y_0\to y_1\to \cdots \to y_k=y$ in $\mathsf{L}$, Equation \eqref{eq:compositionality} implies
\begin{align*}
    \prod_{i=1}^k\pi(y_i|y_{i-1})^t &=
    \pi(y_1|y_0)^t\pi(y_2|y_1)^t\pi(y_3|y_2)^t\cdots \pi(y_k|y_{k-1})^t\\
    &= \pi(y_2|y_0)^t\pi(y_3|y_2)^t\cdots \pi(y_k|y_{k-1})^t\\
    &=\quad \vdots\\
    &= \pi(y_k|y_0)^t.
\end{align*}
Hence the M\"obius coefficients $\zeta_t^{-1}(x,y)$ can be written as
\begin{align*}
    \pi(y|x)^t\sum_{k\geq 0}(-1)^k\#\{\text{nondegenerate paths of length $k$ from $x$ to $y$ in $\mathsf{L}$}\}
\end{align*}
which is equal to $\pi(y|x)^t\zeta_\mathsf{L}^{-1}(x,y)$ by Philip Hall's theorem \citep{Hall36}, cf.~\cite[Proposition 3.8.5]{stanley2011}, \cite[Corollary 1.5]{Leinster2008}.
\end{proof}

Now recall that $\zeta_\mathsf{L}^{-1}(x,y)=\mu_\mathsf{L}(x,y)$, and so $\zeta_t^{-1}$ can be simplified even further by Example \ref{ex:mobius}.

\begin{corollary}\label{cor:zeta_inverse2} 
For any strings $x,y$ in $\ob(\mathsf{L}),$ 
\[\zeta_t^{-1}(x,y)=
\begin{cases}
    -\pi(y|x)^t &\text{if }y\in\mathsf{L}_x^{(1)}\\
    1 &\text{if }y=x\\
    0 &\text{otherwise}.
\end{cases}\]
\end{corollary}

Finally, then, we may use these results to write down the \emph{magnitude function} of the generalized metric space $\mathcal{M}$. By Equation \eqref{eq:magnitude} it is the function $f:(0,\infty)\to\mathbb{R}$ given by $f(t)=\Mag(t\mathcal{M})$, where\footnote{By Corollary \ref{cor:zeta_inverse1}, it follows that $\Mag(t\mathcal{M})=\sum_{x,y\in\ob(\mathcal{M})}\pi(y|x)^t\zeta_{\mathsf{L}}^{-1}(x,y)$, which is a weighted version of the magnitude of the poset $\mathsf{L}$ in Example \ref{ex:mag_of_L}.}
\begin{align*}
    \Mag(t\mathcal{M})&:=\sum_{x,y\in\ob(\mathcal{M})}\zeta_t^{-1}(x,y). 
\end{align*}
Now recall that for any real number $t$, the \emph{$t$-logarithmic entropy} of a probability distribution $p=(p_1,\ldots,p_n)$ is equal to \[H_t(p)=\frac{1}{t-1}\left(1-\sum_{i=1}^np_i^t\right).\]  The following proposition expresses the magnitude function of $\mathcal{M}$ in terms of the logarithmic entropies of the probability distributions $(p_x\colon A\cup\{\dagger\}\to[0,1])_{x\in \ob(\mathsf{L})}$ generated by a  language model and the cardinality of the terminal states $T(\bot)$ of the beginning-of-sentence symbol, which is the set of all  theoretically possible terminating states of the model.  

\begin{proposition}\label{prop:magnitude_L}
The magnitude function of the generalized metric space $\mathcal{M}$ associated with an autoregressive language model is equal to
\[\Mag(t\mathcal{M})=(t-1)\sum_{x\in\ob(\mathcal{M})\setminus T(\bot)}H_t(p_x) + \# (T(\bot)), \qquad t>0.\]
\end{proposition}
\begin{proof}
By Corollary \ref{cor:zeta_inverse2},
\begin{align}
\Mag(t\mathcal{M})&=\sum_{x,y\in\ob(\mathcal{M})}\zeta_t^{-1}(x,y) \nonumber  \\[5pt] 
&=\sum_{x\in\ob(\mathcal{M})} \zeta^{-1}_t(x,x) - \sum_{x\in\ob(\mathcal{M})\setminus T(\bot)}\sum_{y\in\mathsf{L}_x^{(1)}}\pi(y|x)^t \label{eq:intermediate_expr_magM} \\[5pt]
&=\sum_{x\in\ob(\mathcal{M})\setminus T(\bot)}\left(1-\sum_{a\in A\cup\{\dagger\}} p_x(a)^t\right) + \# (T(\bot)). \label{eq:intermediate_expr_magM_final} 
\end{align}
Observe that the first sum in Equation \eqref{eq:intermediate_expr_magM} is over all objects in $\mathcal{M}$, including those strings $x$ that are not the prefix of any other string $y\neq x$, such as all finished texts. The double sum in Equation \eqref{eq:intermediate_expr_magM} accounts for unfinished strings and their extensions by a single token, which is why the indexing set is $\ob(\mathcal{M})\setminus T(\bot).$  Equation \eqref{eq:intermediate_expr_magM_final} follows from the decomposition  $\ob(\mathcal M) = T(\bot) \sqcup (\ob(\mathcal M)\setminus T(\bot))$ and the fact that $\zeta_t^{-1}(x,x)=1$ for all $x\in \ob(\mathcal M)$. 

When $t=1$, Equation \eqref{eq:intermediate_expr_magM_final} is simply $\# (T(\bot))$, since $p_x$ is a probability mass function on $A\cup\{\dagger\}$. When $t\neq 1$, we can multiply the sum by $(t-1)/(t-1)$ and rearrange factors to conclude. 
\end{proof}

As discussed in the Introduction, one might view $\Mag(t\mathcal{M})$ as measuring something of the effective size of a language model's linguistic space, in the sense that for $t\geq 1$, magnitude achieves its smallest possible value if the model is deterministic (that is, it only produces a single preferred text), and it achieves its largest possible value if the model is completely random (at every stage the next-token probability distribution has maximal entropy).

\begin{remark} It is easy to show (e.g.~via L'Hôpital's rule) that if $p:S\to [0,1]$ is a probability mass function on a finite set $S$,
\[ \lim_{t\to 1} H_t(p)= \lim_{t\to 1} \frac{1}{t-1} \left( 1 -\sum_{s\in S} p(s)^t\right) = -\sum_{s\in S} p(s)\ln p(s) =: H(p). \]
The quantity $H(p)$ is the \emph{Shannon entropy} of $p$ (in nats). This entails that Shannon entropy emerges when computing the derivative of the magnitude function $f(t) =\Mag(t\mathcal M)$ at $t=1$:
\begin{equation}\label{eq:derivative-magnitude}
    f'(1) = \sum_{x\in\ob(\mathcal{M})\setminus T(\bot)}H(p_x).
\end{equation}

Alternatively, we can deduce from Equation \eqref{eq:intermediate_expr_magM} that 
$$f(t) = \# (\ob(\mathcal M)) - \sum_{x\in \ob(\mathcal M) \setminus T(\bot)} Z_x(t),$$
where \[Z_x(t) = \sum_{a\in A\cup \{\dagger\}} p_x(a)^t= \sum_{a\in A\cup \{\dagger\}} e^{-t (-\ln p_x(a))} \] is the \emph{partition function} of a system with microstates $A\cup \{\dagger\}$ and internal energy $E_x(a) = -\ln p_x(a) =  d(x,xa)$ at inverse temperature $t>0$. One can verify (see, for instance, \citep{BFL-functor}) that
$$ H(p_x) = -\left.\frac{\mathrm d}{\mathrm d t} Z_x(t)\right|_{t=1},$$
from which Equation \eqref{eq:derivative-magnitude} also follows. 

Since internal energy does not have to be normalized, we could also write $f(t) = \#(\ob(\mathcal M)) - \tilde Z(t)$, where $\tilde Z$ is the partition function of a system with microstates $(x,a)\in S:=(\ob(\mathcal M) \setminus T(\bot))\times (A\cup \{\dagger\}) $  and internal energy $E(x,a) = d(x,a) = -\ln p_x(a)$. The set $S$ parameterizes indecomposable non-identity arrows in $\mathsf L$. For each $t>0$, there is a corresponding probability mass function (Gibbs state) $\rho$ on $S$, given by $\rho(s) = e^{-t E(s)}/\tilde Z(t)$. Then
\begin{equation*}
    \mathbb E_\rho(E) = \sum_{s\in S} E(s) \frac{e^{-tE(s)}}{\tilde Z(t)} = -\frac{\mathrm d}{\mathrm d t} \ln \tilde Z(t) = \frac{f'(t)}{\tilde Z(t)},
\end{equation*}
which gives yet another interpretation of the derivative of the magnitude function. 

\end{remark}

\begin{remark}\label{rmk:subspace_Mx} If desired, one can also make sense of the magnitude function of an enriched category associated with a particular string. That is, suppose $x\in\ob(\mathcal{M})$ and let $\mathcal{M}_x$ denote the $[0,\infty]$-category whose objects are strings $y$ containing $x$ as a prefix. Define the hom-object between any $y,z\in\ob(\mathcal{M}_x)$ to be  $\mathcal{M}_x(y,z):=\mathcal{M}(y,z)=-\ln\pi(z|y)$. Compositionality and the identity requirement hold since they hold in $\mathcal{M}$. So $\mathcal{M}_x$ is indeed a $[0,\infty]$-category and its magnitude function is given by
\[ \zeta_{\mathcal{M}_x}=\zeta_{\mathcal{M}}\big\vert_{\ob(\mathcal{M}_x)\times \ob(\mathcal{M}_x)}
\]
and because of Proposition \ref{prop:vigneaux}, its inverse is equal to
\[ \zeta_{\mathcal{M}_x}^{-1}=\zeta_{\mathcal{M}}^{-1}\big\vert_{\ob(\mathcal{M}_x)\times \ob(\mathcal{M}_x)},
\]
which implies
\begin{align*}
    \Mag(t\mathcal{M}_x) &= \sum_{y,z\in\ob(\mathcal{M}_x)}\zeta_{\mathcal{M}_x}^{-1}(y,z)\\[5pt]
    &= \sum_{y\in\ob(\mathcal{M}_x)}\zeta_{\mathcal{M}_x}^{-1}(y,y) - \sum_{y\in\ob(\mathcal{M}_x)\setminus T(x)} \;\sum_{y\in\mathsf{L}_y^{(1)}}\pi(z|y)^t\\[5pt]
    &=(t-1)\sum_{y\in\ob(\mathcal{M}_x) \setminus T(x)}H_t(p_y) + \#(T( x)).
\end{align*}
\end{remark}

\subsection{Magnitude homology}\label{ssec:homology}

The magnitude homology of metric spaces and enriched categories has been studied extensively by \citet{LeinsterShulman2021}. The main theorem of their work, namely \cite[Theorem 7.14]{LeinsterShulman2021}, expresses the magnitude function of a metric space $(M,d)$ as a weighted sum of Euler characteristics of certain chain complexes. More precisely, for each $\ell\in[0,\infty) $,  there is a chain complex comprised of free abelian groups $(MC_{k,\ell}(M))_{k\in \mathbb N}$ such that $MC_{k,\ell} (M) := \mathbb Z [ G_{k,\ell}]$ where 
\[ G_{k,\ell} = \left\{ ( y_0,\ldots , y_k ) \in  M^{k+1} \mid  
 \sum_{i=0}^{k-1} d(y_i,y_{i+1} ) = \ell \text{ and for all $i$, }y_i \neq y_{i+1}\right\}.\]
 The differential $\partial_k: MC_{k,\ell} (M) \to MC_{k-1,\ell} (M)$ is defined as an alternating sum
 \[\partial_k=\sum_{i=0}^k(-1)^i\partial_k^i\]
where for each $0<i<k,$
\begin{multline*}
    \partial_k^i(y_0,\ldots,y_k ) = \\
\begin{cases}
    (y_0,\ldots,y_{i-1},y_{i+1},\ldots,y_k ) &\text{if } d(y_{i-1},y_i)+d(y_i,y_{i+1})=d(y_{i-1},y_{i+1})\\
    0 &\text{otherwise}
\end{cases}
\end{multline*}
and $\partial_k^0=\partial_k^k=0$. The resulting $[0,\infty)$-graded chain complex $(MC_{\bullet,\ell} (M),\partial_\bullet)$ is called the \emph{magnitude complex} of $M$ \cite[Definition 3.3]{LeinsterShulman2021} and its homology $H_{\bullet, \bullet}(M)$ is called \emph{magnitude homology} \cite[Definition 3.4]{LeinsterShulman2021}. The same definitions hold for any generalized metric space, such as $\mathcal{M}.$

The following corollary expresses the magnitude function of $\mathcal M$ as a weighted sum of Euler characteristics. It is a direct result of Proposition \ref{prop:vigneaux} and also mirrors \cite[Theorem 7.14]{LeinsterShulman2021}, although the situation is considerably simpler here because each complex $(MC_{\bullet,\ell})_\ell$ is bounded and moreover only finitely many $\ell$ can arise as total lengths, which implies we only have to deal with finite sums instead of infinite series.

\begin{proposition}\label{prop:expression_magnitude} For any $t>0$, set $q=e^{-t}$. Then
    \[ \Mag(t\mathcal M) = \sum_{\ell } q^\ell 
\sum_{k\geq 0} (-1)^k \operatorname{rank}(H_{k,\ell}(\mathcal M))
\] 
where 
the first sum is over those finitely many values of $\ell\in [0,\infty)$ for which $\bigcup_{k\geq 0} MC_{k,\ell} (\mathcal M)\neq \emptyset$. 
\end{proposition}

\begin{proof}
    By combining the definition of the magnitude function and its expression in Proposition \ref{prop:vigneaux}, we have
    \begin{align*}
        \Mag(t\mathcal M)& =\sum_{x,y\in \ob(\mathcal M)} \zeta^{-1}_t(x,y) \\
        &=  \sum_{x,y\in \ob(\mathcal M)}\; \sum_{k\geq 0}\sum_{\substack{\text{nondeg. paths}\\ x=y_0\to y_1\to\cdots \to y_k=y}\text{ in } \mathsf L}(-1)^k\prod_{i=1}^k\pi(y_i|y_{i-1})^t.
    \end{align*}
    Observe that the sum over nondegenerate paths can be restricted to those paths $x=y_0\to y_1\to\cdots \to y_k=y$ such that $\prod_{i=1}^k\pi(y_i|y_{i-1}) \neq 0$, which implies that $d(y_i,y_{i+1}) <\infty$ for any $i=0,\ldots,k-1$. 
    Further, there is a bijective correspondence between nondegenerate paths $x=y_0\to y_1\to\cdots \to y_k=y$ of total length $\ell := \sum_{i=0}^{k-1} d(y_i,y_{i+1} ) < \infty$ and generators $(y_0,...,y_k)$ in $G_{k,\ell}$.  Now remark that $\prod_{i=1}^k\pi(y_i|y_{i-1})^t =  \exp(-t \ell)$, and so it is natural to group all the paths that have the same total distance, resulting in the following:
    \begin{equation}\label{eq:expression_mag_chains}
         \Mag(t\mathcal M) = \sum_{k\geq 0} \sum_{\ell\in [0,\infty)} \sum_{c \in G_{k,\ell}} (-1)^k e^{-t\ell}. 
    \end{equation}
    Since $\mathcal M$ is finite, only finitely many real numbers $\ell \in [0,\infty)$ can arise as total lengths of nondegenerate paths, which clarifies the meaning of the  second sum. Moreover \[\sum_{c\in G_{k,\ell} } (-1)^k e^{-t\ell}  =  (-1)^k q^\ell \# G_{k,\ell} = (-1)^k q^{\ell} \operatorname{rank}(MC_{k,\ell}(\mathcal M)) .\] The combination of this fact and Equation \eqref{eq:expression_mag_chains} yields the claim by following a classical argument present in the proof of \cite[Theorem 6.17]{LeinsterShulman2021}, for instance; it is important that each group $MC_{k,\ell}$ is finitely generated and that $(MC_{\bullet, \ell})_\ell$ is bounded. 
\end{proof}

\begin{remark}  Proposition \ref{prop:expression_magnitude} expresses the magnitude function as a weighted sum of Euler characteristics. Together with Proposition \ref{prop:magnitude_L}, it establishes a connection between entropy and topological invariants, perhaps sitting alongside other recent results linking information theory and algebraic topology \citep{BB2015,vigneaux2020,bradley2021,Mainiero2019}.
\end{remark}

Finally, we make a few additional observations about the \emph{magnitude homology} of $\mathcal{M}$, which is the homology of the magnitude complex defined above. Although Leinster and Shulman  work primarily with metric spaces in \citep{LeinsterShulman2021}, two of their results translate directly to our setting: 
\begin{enumerate}
    \item The group $H_{0,0}(\mathcal M)$ is the free abelian group on $\ob (\mathcal M) $, and for any $\ell>0$, $H_{0,\ell}(\mathcal M)=0$, cf.~\cite[Theorem 4.1]{LeinsterShulman2021}.
    \item The group $H_{1,\ell}(\mathcal M)$ is the free abelian group on the set of ordered pairs $( x,y)$ such that $x\neq y$ and $d(x,y)=\ell$ and there does not exist any point strictly between $x$ and $y$ in $\mathsf L$, cf.~\cite[Theorem 4.3]{LeinsterShulman2021}.
\end{enumerate}

Given these facts, one rewrite Equation \eqref{eq:intermediate_expr_magM} as
\begin{equation}\label{eq:magnitude_via_homology}
    \Mag(t\mathcal M) = \operatorname{rank} H_{0,0}(\mathcal M)  - \sum_{\ell\geq 0} q^\ell \operatorname{rank} H_{1,\ell}(\mathcal M). 
\end{equation}
We conjecture that higher homology groups vanish. If this is the case,  Equation \eqref{eq:magnitude_via_homology} would also follow Proposition \ref{prop:expression_magnitude}.

\section{Final remarks}\label{sec:summary}
We conclude with a few remarks surrounding the results of this article. To start, one might notice that the zeta function associated with the generalized metric space $\mathcal{M}$ relates to perplexity, which is frequently used in evaluating autoregressive language models. To elaborate, the \emph{perplexity} of a tokenized sequence $y=a_0 a_1 \cdots a_n$ is 
\[PPL(x)=\exp\left\{-\frac{1}{n}\sum_{i=1}^n\ln p(a_i|y_{<i})\right\}\]
where $p(-|y_{<i})$ is the next-token probability distribution generated by the model when prompted with $y_{<i}$ \cite[Sec. 3.3]{Jurafsky2025}, cf.~Subsection \ref{ssec:induced_probs}. Perplexity is thus equal to the reciprocal of the zeta function $\zeta_t(a_0,y)$ when $t=1/n$, and $a_0$ is the first token (usually $\bot$), and $y$ is the full string:
\[
PPL(y)=1/\zeta_t(a_0,y)=1/e^{t\ln\pi(y|a_0)},
\]
or said differently, $\zeta_t(a_0,y)=1/PPL(y).$  So perhaps for intuition, one might wish to think of the zeta function $\zeta_t(x,y)$ for arbitrary $x,y$ as a generalization of the (reciprocal) of perplexity.

On a different note, as remarked in the introduction, one advantage to defining an enriched category of strings from a language is that one can pass to enriched \textit{copresheaves} on those strings, and the latter functor category contains rich structure. Concretely, there is a $[0,\infty]$-category $\widehat{\mathcal{M}}:=[0,\infty]^\mathcal{M}$ whose objects are $[0,\infty]$-copresheaves, that is, functions $f\colon \mathcal{M}\to[0,\infty]$ satisfying $\max\{f(y)-f(x),0\}\leq \mathcal{M}(x,y)$ for all strings $x$ and $y$. The relevance is that while $\mathcal{M}$ does not, a priori, have any sort of algebraic structure, the functor category $\widehat{\mathcal{M}}$ does, in the sense that one can compute weighted limits and colimits between copresheaves that are reminiscent of logical operations on meaning representations of texts \cite[Section 5]{ECTL2022}. It may be interesting to compute the magnitude of an appropriate finite version of $\widehat{\mathcal{M}}$, cf.~\cite[Example 2.5]{Leinster2008}.

Lastly, let us draw a brief connection to the notion of diversity. \emph{Categorical diversity} \citep{Chen2023} depends on a finite category $\mathsf{A}$, a probability mass function $p\colon\ob(\mathsf{A})\to[0,1]$, and a similarity matrix $\theta\colon \ob(\mathsf{A})\times \ob(\mathsf{A})\to[0,\infty)$ such that
\[\theta(a,b)=0 \quad \text{if } \quad \mathsf{A}(a,b)=\varnothing.\]
In the case of a generalized metric space, $\theta=\zeta_t$ and diversity is given by
\begin{align*}
    \mathscr{H}(\mathcal{M},p,\zeta)& = -\sum_{y\in\ob(\mathcal{M})}p(y)\log\left(\sum_{z\in\ob(\mathcal{M})}\zeta_t(y,z)p(z)\right)\\
   &=  - \sum_{y\in\ob(\mathcal{M})}p(y) \log\left(\sum_{z\in\ob(\mathcal{M})}\pi(z|y)^tp(z)\right).
\end{align*}
In  \citep{Chen2023} the authors make a case for seeing this function as a probabilistic extension of $\log(\text{magnitude})$, just as entropy is a probabilistic extension of $\log(\text{cardinality})$ already in the work of Boltzmann and Gibbs. In the particular case of metric spaces, there are very deep connections between magnitude and diversity \citep{Meckes2015}. We leave a detailed study of diversity of enriched categories of texts for future work, but remark here that if we take $p = \pi(-|x)|_{T(x)}$, then we recover the Shannon entropy of $\pi(-|x)|_{T(x)}$, which unlike our expression for the magnitude does not treat all states equally and gives more weight to those that are highly probable. 

Very interestingly, in the particular case of a homogeneous and irreducible Markov chain with transition matrix $P:A\times A\to [0,1]$ and with stationary distribution $p_{\bot}$, in the sense of Example \ref{ex:markov_chain},  one can verify (following a standard computation, see e.g.~the proof of \cite[Thm. 4.27]{Walters1982}) that the \emph{diversity rate}
\begin{multline}
    \lim_{n\to \infty} \frac{\mathscr{H}(\mathcal{M},\pi(-|\bot)|_{T(\bot)},\zeta)}{n} \nonumber \\ = \lim_{n\to\infty} -\frac 1n \sum_{(a_1,...,a_{N-1}) \in A^{N-1}} \pi(\bot a_1\cdots a_{N-1}|\bot) \log \pi(\bot a_1\cdots a_{N-1}|\bot)
\end{multline}
equals the Kolmogorov--Sinai entropy (rate) of the Markov chain,
$$-\sum_{a\in A} \sum_{a'\in A} p_{\bot}(a) P(a,a') \log P(a,a'), $$
which plays a fundamental role in information theory \cite[App. 3]{Shannon1948} and is an asymptotic  lower bound for the logarithm of the perplexity of any LM that approximates the Markov chain \cite[Sec. 3.7]{Jurafsky2025}.

\end{document}